\documentclass[11pt,reqno,a4paper]{article}
\usepackage{comment}

\usepackage{euscript}
\usepackage{amsmath,amssymb,amsfonts,amsthm}
\usepackage{mathtools}
\usepackage{fullpage}
\usepackage{xcolor}
\usepackage{dsfont}
\usepackage[colorlinks=true, linkcolor=blue]{hyperref}
\usepackage{sectsty}

\newtheoremstyle{special}%
{}%
{}%
{}%
{}%
{\scshape}%
{.}%
{.5em}%
{}

\newtheorem{maintheorem}{Theorem}
\newtheorem{theorem}{Theorem}

\newtheorem{lemma}[theorem]{Lemma}

\newtheorem{definition}[theorem]{Definition}

\theoremstyle{special}
\newtheorem{remark}[theorem]{Remark}

\allsectionsfont{\centering\mdseries\normalsize\scshape}

\def\N{\mathbb{N}}

\def\R{\mathbb{R}}
\def\B{\mathcal{B}}

\def\om{\omega}

\DeclareMathOperator{\var}{Var}

\DeclareMathOperator{\essinf}{essinf}

\DeclareMathOperator{\Lip}{Lip}
\DeclareMathOperator{\dist}{dist}

\author{D. Dragi\v cevi\' c \footnote{Faculty of Mathematics, University of Rijeka, Rijeka, Croatia. {\tt E-mail: ddragicevic@math.uniri.hr}.}\and 
	Y. Hafouta\footnote{Department of Mathematics, Ben-Gurion University, Israel and Department of Mathematics, The University of Florida, USA. {\tt E-mail: yeor.hafouta@mail.huji.ac.il}.}}

\begin{document}

\title{Quenched and annealed linear response for some partially hyperbolic skew products}
\maketitle

\begin{abstract}
We prove quenched and annealed statistical stability, linear response, and differentiability of asymptotic moments for parametric families of partially hyperbolic skew products, with random hyperbolic maps on the fibers. 
The main novelty is that the base maps also depend on the parameter, which leads to different formulas in the linear response and the derivative of the asymptotic moments with respect to the parameter. Our annealed results apply to partially hyperbolic maps that are not covered in \cite{BashCastro26,Dol,DS}. 
\end{abstract}

\section{Introduction}
Let $X$ be a topological space and $(\tau_\epsilon)_{\epsilon \in I}$ a family of sufficiently regular maps $\tau_{\epsilon}\colon X\to X$, where $I\subset\R$ such that $0\in I$ is an accumulation point. Here, we  view $\tau_\epsilon$ as a ``sufficiently small'' perturbation of $\tau_0$. Suppose that for each $\epsilon \in I$, $\tau_\epsilon$ admits a unique physical measure $\mu_\epsilon$. The problem of linear response is concerned with the regularity of the map $\epsilon \to \mu_\epsilon$ at $0$. More precisely, let $\mathcal O$ be a ``large'' class consisting of real-valued observables $\varphi \colon X\to \mathbb R$. We say that a family $(\tau_\epsilon)_{\epsilon \in I}$ exhibits:
\begin{itemize}
\item \emph{statistical stability} if the map $\epsilon \to \int_X \varphi \, d\mu_\epsilon$ is continuous at $0$ for each $\varphi \in \mathcal O$;
\item \emph{linear response} if the map $\epsilon \to \int_X \varphi\, d\mu_\epsilon$ is differentiable at $0$ for each $\varphi \in \mathcal O$.
\end{itemize}
We note that in some cases the measures $\mu_\epsilon$ can be identified as elements of a certain Banach space $\B$, and in that case it makes sense to ask whether the map $\epsilon \mapsto \mu_\epsilon$ is continuous or differentiable in $0$ when viewed as a map from $I$ to $\mathcal B$.

The question of statistical stability is motivated by physics, where one only observes a small perturbation $\tau_\epsilon$ of the ideal system $\tau_0$, and it is important that the physical measure $\mu_\epsilon$ is continuous in $\epsilon$. The same reasoning applies to differentiability, which  is also important
in averaging, rigidity  and statistical physics (see \cite{[8],[31],[35],[50]}). 
 We also refer to \cite[Corollary 2.2]{Dol} for applications of linear response to certain refined versions of the mean ergodic theorem and the central limit theorem.

The literature dealing with  linear response is vast.  More precisely,  linear response (or the lack of it) has been discussed for smooth expanding 
systems~\cite{B2,Baladibook,S}, piecewise expanding maps of the interval \cite{B1,BS1}, unimodal maps~\cite{BS2, LS}, intermittent maps \cite{BahSau,BT, L, K}, smooth hyperbolic diffeomorphisms and flows~\cite{BL1, BL2, GL, Ruelle97}, large classes of $C^\infty$ partially hyperbolic systems~\cite{Dol}, systems with cusps~\cite{BG}, as well as discontinuous perturbations of hyperbolic systems~\cite{C}. We refer to~\cite{B2} for a detailed survey of the linear response theory for deterministic dynamical systems which has many interesting applications, for instance to the continuity and differentiability of the variance in the central limit theorem (CLT) for suitable observables (see for example~\cite{bomfim2016, GKLM}).

Let $\Omega$ and $M$ be two bounded Riemannian manifolds and let $X=\Omega\times M$.
In this paper, we consider parametric families of skew products of the form
\[
\tau_\epsilon(\omega, x)=(\sigma_\epsilon \omega, T_{\omega, \epsilon}(x)), \quad (\omega, x)\in \Omega \times M,
\]
where the maps $T_{\om,\epsilon}\colon M\to M$ are hyperbolic maps such that for each $\epsilon$, the family of maps $(T_{\omega, \epsilon})_{\omega \in \Omega}$ admits a family of equivariant physical measures $h_{\om,\epsilon}$, namely $(T_{\om,\epsilon})_*h_{\om,\epsilon}=h_{\sigma_\epsilon\om,\epsilon}$. 
When the maps $T_{\om,\epsilon}$ are expanding, the measures $h_{\om,\epsilon}$ are absolutely continuous, while in the hyperbolic case they are random SRB states in the sense that their random basin has positive Riemannian volume.
Such skew products are classical examples of partially hyperbolic systems, where, in a sense, $\Omega$ plays the role of the central  direction. Throughout this paper, we
 will always assume that $\sigma_\epsilon \om$ is continuous in $(\om,\epsilon)$, $\sigma_0^{-1}$ is Lipschitz, and $\sigma_\epsilon$ preserves a probability measure $\mathbb P_\epsilon$.

In this context, there are two types of statistical stability and a linear response. The first one are the ones presented above with the measure 
$$
\mu_\epsilon=\int h_{\om,\epsilon}\, d\mathbb P_\epsilon(\om)
$$
which is invariant under $\tau_\epsilon$. Such results will be referred to as \textit{annealed statistical stability} and \textit{annealed linear response}. 
The second type is the quenched ones. 
We say that the parameterized family of maps $(T_{\omega, \epsilon})_{\omega \in \Omega}$, $\epsilon \in I$ exhibits:
\begin{itemize}
\item \emph{quenched statistical stability} 
if the map $\epsilon \mapsto \int_M\varphi \, d h_{\om, \epsilon}$ is continuous at $0$, where $\varphi \colon M\to \R$ belongs to a suitable class of observables;
\item \emph{quenched linear response} if the map $\epsilon \mapsto \int_M\varphi \, d h_{\om, \epsilon}$ is differentiable at $0$, where $\varphi \colon M\to \R$ belongs to a suitable class of observables;
\end{itemize}

We first prove quenched statistical stability when the transfer operators of the maps $T_{\om,\epsilon}$ are continuous in the weak sense at points of the form $(\om,0)$.
Under appropriate continuity assumptions on $\epsilon\to \mathbb P_\epsilon$ at $\epsilon=0$ (i.e. statistical stability for $\mathbb P_\epsilon$) we prove annealed statistical stability. For example, we can consider partially hyperbolic base maps $\sigma_\epsilon$ like the ones in \cite{Dol} and SRB measures $\mathbb P_\epsilon$.

Regarding the linear response, the quenched result requires that the transfer operators of $T_{\om,\epsilon}$  satisfy some differentiability condition in $\epsilon$ and $\om$ (in a weak sense) at points of the form $(\om,0)$ and that
\begin{equation}\label{Contt}
\text{Lip}(\sigma^{-k})\leq Ce^{ck}, \quad  k\geq 0,\,\,\sigma:=\sigma_0    
\end{equation}
for $c$ small enough.    
This essentially means that, in the $\Omega$-direction, $\tau_0$ cannot contract too fast. 
This, of course, includes examples where $\sigma$ (or $\partial\sigma$) is an isometry (e.g. group ``rotation" with left invariant metric), and in that case $\Omega$ is the central direction of $\tau_0$. This case is probably one of the simplest types of partially hyperbolic skew products. Note  that \eqref{Contt} is an open condition in an appropriate sense (namely, it is stable under appropriate small perturbations). 
Notice that when $\sigma^{-1}$ is Lipchitz continuous, then we can always ensure that \eqref{Contt} holds by replacing the distance $d$ in $\Omega$ with a sufficiently small power $\alpha>0$ of $d$. In applications we  assume that $\sup_\om d(\sigma_\epsilon^{-1}\om, \sigma^{-1}\om)=O(\epsilon)
$ and the latter change of metric means that we can replace $c$ in \eqref{Contt} with $\alpha c$ if $\sup_\om d(\sigma_\epsilon^{-1}\om, \sigma^{-1}\om)=O(\epsilon^{1/\alpha})$.


We also obtain an annealed linear response under an additional assumption about linear response for the maps $\epsilon\to\mathbb P_\epsilon$, which holds true for group rotations (since then $\mathbb P_\epsilon=\text{Haar}$), but also for small perturbations of circle rotations (which also satisfy \eqref{Contt}); see \cite{GaSo}.

We also prove the differentiability of the asymptotic moments with respect to $\epsilon$. This means that we consider sufficiently regular functions $f_\epsilon:\Omega\times M\to\mathbb R$, and then we prove that the quantities 
$$
M_{k,\epsilon}=\lim_{n\to\infty}n^{-k/2}h_{\om,\epsilon}\left((\bar S_{n,\epsilon}^\omega f)^k\right)
$$
are differentiable at $\epsilon=0$. Here  
$$
S_{n,\epsilon}^\omega f=\sum_{j=0}^{n-1}f_{\epsilon}(\sigma_\epsilon^j\om,T_{\om,\epsilon}^j),\,\,\, \bar S_{n,\epsilon}^\omega f=S_{n,\epsilon}^\omega f-h_{\om,\epsilon}(S_{n,\epsilon}^\omega f)
$$
where $T_{\om,\epsilon}^j=T_{\sigma_\epsilon^{j-1}\om,\epsilon}\circ\ldots\circ T_{\sigma_\epsilon\om,\epsilon}\circ T_{\om,\epsilon}$. 
 These are the asymptotic moments in the central limit theorem for the sequence $(n^{-1/2}\bar S_{n,\epsilon}^\om f)$, when viewed as random variables on $(M,h_{\om,\epsilon})$.
Note that $M_{k,\epsilon}=M_{2,\epsilon}^{k/2}\mathbb E[Z^k]$, where $Z$ is standard normal. Thus, it is enough to prove the differentiability when $k=2$.
Our approach  is based on proving the differentiability of $M_{2,\epsilon}$  using the Green-Kubo formula. This requires proving regularity conditions of the map $\om\to h_{\omega,0}$ which are discussed in Section \ref{RegSec}.

\subsection{A comparison of our results with existing literature}
Our results have two main innovations. First, from a random dynamics point of view, we allow the base map $\sigma_\epsilon$ to depend on $\epsilon$. Second, from a deterministic point of view, our annealed results allow us to consider classes of partially hyperbolic maps that are not covered in \cite{Dol}, and parametric families of partially hyperbolic skew products that are not covered in  \cite{BashCastro26, CN, DS, DH CMP}. All of this will be elaborated on in the following.
\subsubsection{Annealed results}

In \cite{Dol}, a comprehensive study of the linear response was conducted for SRB measures $\mu_\epsilon$ of partially hyperbolic maps $\tau_\epsilon$. In that setup, after a reduction, the setup includes a $C^\infty$ partially hyperbolic map $\tau_0$ with dominated splitting and a central direction $E_c$ such that $\partial \tau_0|E_c$ is an isometry. For our maps, the support $\Omega_\epsilon$ of $\mathbb P_\epsilon$  plays the role of the central manifold, and condition \eqref{Contt} is satisfied when $\sigma$ or $\partial\sigma$ are isometries. However, we also assume that $\epsilon\to\mathbb P_\epsilon$ satisfies a linear response in an appropriate sense. Moreover, when the fiber maps $T_{\om,\epsilon}$ expand, we only assume some degree of smoothness, whereas in the hyperbolic case we work we assume smoothness of higher orders but less than $C^\infty$. In conclusion, our results concern a  different class of partially hyperbolic maps than in \cite{Dol}.

As noted above, in \cite{DS} (see also~\cite{CN}) the linear response was obtained in the setup of our article but in the special case where $\sigma_\epsilon=\sigma_0$ do not depend on $\epsilon$.
In \cite{BashCastro26} the linear response was studied for skew products with not necessarily invertible base maps $\sigma_\epsilon$ and uniformly contractive maps $T_{\om,\epsilon}$. However, due to possible expansion in the $M$ direction, we have different requirements on the base maps $\sigma_\epsilon$.
Another related result was obtained in~\cite[Theorem B]{DH CMP}, where the linear response for skew products was obtained for non-uniformly expanding random maps $T_{\om,\epsilon}$ and with a  fixed base map $\sigma_\epsilon=\sigma_0$.

\subsubsection{Quenched results}
So far, the study of the quenched linear response has focused on the case when the base map $\sigma_\epsilon$ does not depend on $\epsilon$. 
In that case, the study was initiated by  Rugh and Sedro~\cite{RS} for random expanding dynamics, followed by the works by Dragi\v cevi\' c and Sedro~\cite{DS} and Crimmins and  Nakano~\cite{CN} for random (partially) hyperbolic dynamics. More recently, in~\cite{DGTS, DL}, the authors established a quenched linear response for a class of random intermittent maps.

On the other hand, Dragi\v cevi\' c, Giulietti and Sedro~\cite{DGS} established the quenched linear response for a class of random dynamics that exhibits nonuniform decay of correlations. 
 More precisely, they considered the case of cocycles that are expanding on average. Namely, 
 in~\cite{DGS} it is assumed that there exists a $\log$-integrable random variable $\underline{\gamma} \colon \Omega \to (0, \infty)$ such that $\gamma_{\om, \epsilon} \ge \underline{\gamma}$ and 
\begin{equation}\label{minexp}
\int_\Omega \log \underline{\gamma}(\om) \, d\mathbb P(\om)>0,
\end{equation}
where $\gamma_{\om, \epsilon}$ denotes the minimal expansion of $T_{\om, \epsilon}$.  Note that~\eqref{minexp} allows for $\gamma_{\om, \epsilon}<1$ on a set of positive measure. Thus, in sharp contrast to~\cite{RS}, it is not required that all maps $T_{\om, \epsilon}$ expand or that there exists a uniform (in $\om$) lower bound for minimal expansion. 
It was also illustrated in~\cite[Appendix A]{DGS} that in this setup, the annealed linear response may fail even if the quenched linear response holds. Finally, in \cite{DH CMP} the authors showed that under some mixing assumptions on the base map $\sigma$ when the maps are expanding but not uniformly in $\omega$ then the effective quenched linear response holds (this has applications to the annealed linear response, discussed in the previous section).

As noted above, our main contribution compared to the above results is that we allow the base map to depend on $\epsilon$.

\subsection{Structure of the paper}
In Section \ref{SS} we prove quenched and annealed statistical stability for parametric random operator satisfying some abstract conditions. In Section \ref{LRsec} we prove the quenched and annealed linear response under similar abstract conditions. Section \ref{RegSec} concerns 
the regularity of the maps $(\om,\epsilon)\mapsto h_{\omega,\epsilon}$.
In Section \ref{SecDiff} we will prove the differentiability of asymptotic moments, while in Section \ref{Examples} we discuss applications to random expanding and hyperbolic maps and certain classes of base maps $\sigma_\epsilon$.


\section{Statistical stability}\label{SS}
Let $\Omega$ be a bounded metric space. Take a set $I\subset\mathbb R$ containing $0$, with $0$ being an accumulation point of $I$, and suppose that for each $\epsilon \in  I$, we have a homeomorphism $\sigma_\epsilon \colon \Omega \to \Omega$ and a Borel probability measure $\mathbb P_\epsilon$ on $\Omega$ which is $\sigma_\epsilon$-invariant. Henceforth, we denote $\sigma=\sigma_0$.

We consider two Banach spaces $(\mathcal B_w, \| \cdot \|_w)$ and $(\mathcal B_s, \| \cdot \|_s)$ such that $\mathcal B_s$ is embedded in $\mathcal B_w$ and $\|\cdot \|_w\le \|\cdot \|_s$ on $\mathcal B_s$. Moreover, let $\psi\in \mathcal B_s'$ be a non-zero bounded functional on $\mathcal B_s$   that admits a bounded extension to $\mathcal B_w$. Assume that for each $\omega \in \Omega$ and $\epsilon \in I$, $\mathcal L_{\omega, \epsilon}$ is a bounded linear operator on both $\mathcal B_s$ and $\mathcal B_w$. For $\omega \in \Omega$, $\epsilon \in I$ and $n\in \mathbb N$, set 
\begin{equation}\label{compos}
\mathcal L_{\omega, \epsilon}^n:=\mathcal L_{\sigma_\epsilon^{n-1}\omega, \epsilon}\circ \ldots \circ \mathcal L_{\sigma_\epsilon  \omega, \epsilon}\circ \mathcal L_{\omega, \epsilon}.
\end{equation}
We will write $\mathcal L_\omega$ instead of $\mathcal L_{\omega, 0}$ and similarly $\mathcal L_{\om}^j$ instead of $\mathcal L_{\om,0}^j$.
We also assume that 
\begin{equation}\label{preserv}
    \psi(\mathcal L_{\omega, \epsilon}h)=\psi(h), \quad \text{for $\epsilon \in I$, $\omega \in \Omega$ and $h\in \mathcal B_w$.}
\end{equation}
In our applications, $\mathcal L_{\omega, \epsilon}$ will be a transfer operator associated to a map $T_{\omega, \epsilon}\colon M\to M$, where $M$ is a compact Riemannian manifold.
\subsection{Quenched statistical stability}
We begin by establishing a quenched statistical stability result.
 \begin{maintheorem}\label{T1}
Assume that there exist constants $C, \lambda>0$, $\sigma_\epsilon$-invariant measurable sets $\Omega_\epsilon\subset\Omega,\,\epsilon\in I$ such that $\mathbb P_\epsilon(\Omega_\epsilon)=1$  and for each $\epsilon \in I$ a family $(h_{\omega, \epsilon})_{\omega \in \Omega_\epsilon}\subset \mathcal B_s$ such that:
\begin{enumerate}
\item for $h\in \mathcal B_w$ with $\psi(h)=0$, $n\in \mathbb N$, $\epsilon \in I$ and $\omega \in \Omega_\epsilon$,
\begin{equation}\label{dec}
\|\mathcal L_{\omega, \epsilon}^n h\|_{w}\le Ce^{-\lambda n}\|h\|_w;
\end{equation}
\item for $\epsilon \in I$, $h\in \mathcal B_s$, $\omega\in \Omega_\epsilon$ and  $\om'\in\Omega_0$,
\begin{equation}\label{triple}
\|\mathcal L_{\omega, \epsilon}h-\mathcal L_{\omega'}h\|_w \le C(d(\omega, \omega')^{\zeta}+|\epsilon|^\zeta)\|h\|_s
\end{equation}
for some constant $\zeta\in(0,1]$;
\item we have $\psi(h_{\omega, \epsilon})=1$ and 
\begin{equation}\label{bound}
\|h_{\omega, \epsilon}\|_{s}\le C \quad \text{for $\omega \in \Omega_\epsilon$ and $\epsilon \in I$;}
\end{equation}
\item for $\epsilon \in I$ and $\omega \in \Omega_\epsilon$,
\[
\mathcal L_{\omega, \epsilon}h_{\omega, \epsilon}=h_{\sigma_\epsilon \omega, \epsilon};
\]
\item $\sigma^{-1}$ is Lipschitz continuous. 
\end{enumerate}
Then there exists $D>0$ such that 
\[
\|h_{\omega, \epsilon}-h_\omega\|_w \le D\left(|\epsilon|^\zeta+\left(\dist_{C^0}(\sigma^{-1}, \sigma_\epsilon^{-1})\right)^{\alpha\zeta}\right),
\]
for each $\omega \in \Omega_\epsilon \cap \Omega_0$, where $h_\omega:=h_{\omega, 0}$,  $\dist_{C^0}(\sigma^{-1}, \sigma_\epsilon^{-1}):=\max_{y\in \Omega}d(\sigma^{-1}(y), \sigma_\epsilon^{-1}(y))$, and $\alpha\in(0,1]$ is a positive number satisfying $c\alpha\zeta<\lambda$, where $c=\ln(\Lip(\sigma^{-1}))$ and $\Lip(\cdot)$ denotes the Lipschitz constant. 
\end{maintheorem}
\begin{remark}\label{Rem}
 In the applications in Section \ref{Examples} for random hyperbolic or expanding maps, we can actually take $\Omega_\epsilon=\Omega$ for each $\epsilon \in I$.
 We note that when all $\mathbb P_\epsilon$ are absolutely continuous with respect to some positive measure $\nu$, we can require that all conditions hold for $\om\in\Omega'$ for a set $\Omega'$ such that $\nu(\Omega\setminus\Omega')=0$. 
 \end{remark}

 \begin{proof}[Proof of Theorem \ref{T1}]
For $\omega \in \Omega_\epsilon\cap \Omega_0$ we have 
\[
\begin{split}
h_{\omega, \epsilon}-h_\omega &=\mathcal L_{\sigma_\epsilon^{-1}\omega, \epsilon}h_{\sigma_\epsilon^{-1}\omega, \epsilon}-\mathcal L_{\sigma^{-1}\omega}h_{\sigma^{-1}\omega}\\
&=\mathcal L_{\sigma_\epsilon^{-1}\omega, \epsilon}(h_{\sigma_\epsilon^{-1}\omega, \epsilon}-h_{\sigma^{-1}\omega})+(\mathcal L_{\sigma_\epsilon^{-1}\omega, \epsilon}-\mathcal L_{\sigma^{-1}\omega})h_{\sigma^{-1}\omega}.
\end{split}
\]
Analogously,
\[
h_{\sigma_\epsilon^{-1}\omega, \epsilon}-h_{\sigma^{-1}\omega}=\mathcal L_{\sigma_\epsilon^{-2}\omega, \epsilon}(h_{\sigma_\epsilon^{-2}\omega, \epsilon}-h_{\sigma^{-2}\omega})+(\mathcal L_{\sigma_\epsilon^{-2}\omega, \epsilon}-\mathcal L_{\sigma^{-2}\omega})h_{\sigma^{-2}\omega},
\]
and thus
\[
h_{\omega, \epsilon}-h_\omega=\mathcal L_{\sigma_\epsilon^{-2}\omega, \epsilon}^2(h_{\sigma_\epsilon^{-2}\omega, \epsilon}-h_{\sigma^{-2}\omega})+\mathcal L_{\sigma_\epsilon^{-1}\omega, \epsilon}(\mathcal L_{\sigma_\epsilon^{-2}\omega, \epsilon}-\mathcal L_{\sigma^{-2}\omega})h_{\sigma^{-2}\omega}+(\mathcal L_{\sigma_\epsilon^{-1}\omega, \epsilon}-\mathcal L_{\sigma^{-1}\omega})h_{\sigma^{-1}\omega}.
\]
Proceeding in the same manner, we have 
\[
h_{\omega, \epsilon}-h_\omega=\mathcal L_{\sigma^{-n}_\epsilon \om, \epsilon}^n(h_{\sigma_\epsilon^{-n}\omega, \epsilon}-h_{\sigma^{-n}\omega})+\sum_{j=0}^{n-1}\mathcal L_{\sigma_\epsilon^{-j}\omega, \epsilon}^j(\mathcal L_{\sigma^{-(j+1)}_\epsilon \omega, \epsilon}-\mathcal L_{\sigma^{-(j+1)}\omega})h_{\sigma^{-(j+1)}\omega},
\]
for $n\in \mathbb N$.   Using~\eqref{dec} and~\eqref{bound} (note that $\psi(h_{\omega, \epsilon}-h_\omega)=0$), we have 
\[
\|\mathcal L_{\sigma^{-n}_\epsilon \om, \epsilon}^n(h_{\sigma_\epsilon^{-n}\omega, \epsilon}-h_{\sigma^{-n}\omega})\|_w \le Ce^{-\lambda n}\|h_{\sigma_\epsilon^{-n}\omega, \epsilon}-h_{\sigma^{-n}\omega}\|_w\le 2C^2e^{-\lambda n},
\]
and consequently
\[
\|\mathcal L_{\sigma^{-n}_\epsilon \om, \epsilon}^n(h_{\sigma_\epsilon^{-n}\omega, \epsilon}-h_{\sigma^{-n}\omega})\|_w\to 0 \quad \text{when $n\to \infty$,}
\]
for $\omega \in \Omega_\epsilon \cap \Omega_0$. Thus, 
\begin{equation}\label{1314}
h_{\omega, \epsilon}-h_\omega=\sum_{j=0}^\infty \mathcal L_{\sigma_\epsilon^{-j}\omega, \epsilon}^j(\mathcal L_{\sigma^{-(j+1)}_\epsilon \omega, \epsilon}-\mathcal L_{\sigma^{-(j+1)}\omega})h_{\sigma^{-(j+1)}\omega}, \quad \omega \in \Omega_\epsilon \cap \Omega_0.
\end{equation}
We now have (using~\eqref{dec}, \eqref{triple} and~\eqref{bound}) that 
\[
\begin{split}
\|h_{\omega, \epsilon}-h_\omega\|_w &\le \sum_{j=0}^\infty \|\mathcal L_{\sigma_\epsilon^{-j}\omega, \epsilon}^j(\mathcal L_{\sigma^{-(j+1)}_\epsilon \omega, \epsilon}-\mathcal L_{\sigma^{-(j+1)}\omega})h_{\sigma^{-(j+1)}\omega}\|_w \\
&\le C^3\sum_{j=0}^\infty e^{-\lambda j}(d(\sigma_\epsilon^{-(j+1)}\omega, \sigma^{-(j+1)}\omega)^\zeta+|\epsilon|^\zeta)\\
&=\frac{C^3}{1-e^{-\lambda}}|\epsilon|^\zeta+C^3\sum_{j=0}^\infty e^{-\lambda j}d(\sigma_\epsilon^{-(j+1)}\omega, \sigma^{-(j+1)}\omega)^\zeta,
\end{split}
\]
for $\omega \in \Omega_\epsilon \cap \Omega_0$, as by~\eqref{preserv} \[\psi((\mathcal L_{\sigma^{-(j+1)}_\epsilon \omega, \epsilon}-\mathcal L_{\sigma^{-(j+1)}\omega})h_{\sigma^{-(j+1)}\omega})=0.\]
Next, note that 
\begin{equation}\label{1223}
\begin{split}
d(\sigma^{-j}\omega,\sigma_{\epsilon}^{-j}\omega) &\leq \sum_{k=0}^{j-1}d\left(\sigma^{-k}(\sigma^{-1}(\sigma_{\epsilon}^{-(j-(k+1))}\om)),\sigma^{-k}(\sigma_{\epsilon}^{-1}(\sigma_{\epsilon}^{-(j-(k+1))}\om))\right)\\
&\le \sum_{k=0}^{j-1}\Lip (\sigma^{-k})d\left(\sigma^{-1}(\sigma_{\epsilon}^{-(j-(k+1))}\om), \sigma_{\epsilon}^{-1}(\sigma_{\epsilon}^{-(j-(k+1))}\om)\right)\\
&\le \sum_{k=0}^{j-1}e^{ck}\dist_{C^0}(\sigma^{-1}, \sigma_\epsilon^{-1})\leq C'e^{cj}\dist_{C^0}(\sigma^{-1}, \sigma_\epsilon^{-1}),
\end{split}
\end{equation}
where $C'>0$ is independent of $\omega$, $\epsilon$ and $j$.
Using that the metric on $\Omega$ is bounded by some constant $D_0>0$, we conclude that
$$
d(\sigma^{-j}\omega,\sigma_{\epsilon}^{-j}\omega)^{\zeta}\leq D_0^{(1-\alpha)\zeta}(C')^{\alpha\zeta} e^{cj\alpha\zeta}\left(\dist_{C^0}(\sigma^{-1}, \sigma_\epsilon^{-1})\right)^{\alpha\zeta}.
$$
Consequently,
\[
\|h_{\omega, \epsilon}-h_\omega\|_w\le \frac{C^3}{1-e^{-\lambda}}|\epsilon|^\zeta+C''\left(\dist_{C^0}(\sigma^{-1}, \sigma^{-1}_\epsilon)\right)^{\alpha\zeta}\sum_{j=0}^\infty e^{-\lambda j+c\alpha\zeta(j+1)},
\]
where $C''>0$ is independent of $\omega$ and $\epsilon$.
We conclude that 
\[
\|h_{\omega, \epsilon}-h_\omega\|_w \le D\left(|\epsilon|^\zeta+\left(\dist_{C^0}(\sigma^{-1}, \sigma_\epsilon^{-1})\right)^{\alpha\zeta}\right),
\]
for each $\omega \in \Omega_\epsilon\cap \Omega_0$, where $D>0$ is a constant independent of $\omega$ and $\epsilon$.
 \end{proof}

  \subsection{Annealed statistical stability}
Let $(\mathfrak F,\|\cdot\|_{\mathfrak F})$ be a semi-normed space of measurable bounded real-valued functions on $\Omega$. We suppose that 
$$
R:=\sup_{\epsilon\in I}\sup_{f\in\mathfrak F: \|f\|_{\mathfrak F}\leq 1}\left|\int_\Omega f\,d\mathbb P_\epsilon\right|<\infty.
$$
For a finite Borel measure $\nu$ on $\Omega$ denote 
$$
\|\nu\|_{\mathfrak F}=\sup_{f\in\mathfrak F, \|f\|_{\mathfrak F}\leq 1}\left|\int_\Omega f d\nu\right|.
$$

Let $M$ be a compact Riemannian manifold. In order to include applications to small perturbations of Anosov maps $T_{\om,\epsilon}$ as well as to expanding maps $T_{\om,\epsilon}$, we will consider two possibilities. The first possibility is that $\mathcal B_w$ contains bounded measurable functions on $M$ and $\|\cdot\|_\infty\leq \|\cdot\|_w$. In this case, for bounded measurable functions $g,h \colon M\to\mathbb R$, we write $h(g)=\int_M g h\,d m$, where $m$ is the normalized volume measure on $M$. The second possibility is that there exists $r\in \mathbb N$ such that the elements of $\mathcal B_w$ are distributions or order $r$, i.e., there is $C>0$ such that for $h\in \mathcal B_w$,
\begin{equation}\label{distrib}
|h(\varphi)|\le C\|h\|_w \cdot \|\varphi \|_{C^r},
\end{equation}
where $\varphi \in C^r(M, \mathbb C)$.

We have the following annealed statistical stability result.
 \begin{maintheorem}\label{T2}
Let the assumptions of Theorem~\ref{T1} be in force  
and assume that there\footnote{c.f. Remark \ref{Rem}} is a measurable set $\Omega'$ such that $\Omega'\subset\Omega_\epsilon$  for every $\epsilon$ and $\mathbb P_\epsilon(\Omega')=1$. In addition,  suppose that 
$h_{\omega, \epsilon}$ are Borel probability measures on $M$ (or densities when $\|\cdot\|_\infty\leq \|\cdot\|_w$) that depend measurably on $\omega$ for each $\epsilon \in I$, 
\begin{equation}\label{1211}
\lim_{\epsilon \to 0}\|\mathbb P_\epsilon-\mathbb P_{0}\|_{\mathfrak F}=0
 \text{ and } \quad \lim_{\epsilon \to 0}\dist_{C^0}(\sigma^{-1}, \sigma_\epsilon^{-1})=0.
\end{equation}
Finally, let $\Phi \colon \Omega \times M\to \mathbb R$ be measurable such that $\Phi(\omega, \cdot)\in C^r(M, \mathbb R)$ (or $\Phi(\omega, \cdot)\in L^1(m)$ when $\|\cdot\|_\infty\leq \|\cdot\|_w$) and that 
the functions $\varphi, \theta \colon \Omega\to\mathbb R$ given by
\[
\varphi(\omega):=h_\omega(\Phi(\omega, \cdot)), \,\theta(\omega):=\|\Phi(\omega, \cdot)\|_{C^r}, \quad \,\,\left (\text{or }\,\theta(\om)=\int_M |\Phi(\om,\cdot)|dm\text{ when }\|\cdot\|_\infty\leq \|\cdot\|_w \right )
\]
belong to $\mathfrak F$.
Then
\[
\lim_{\epsilon \to 0}\int_{\Omega \times M}\Phi \, d\mu_\epsilon=\int_{\Omega \times M}\Phi \, d\mu,
\]
where $\mu_\epsilon$ is the measure on $\Omega \times M$ such that 
\[
\mu_\epsilon(A\times B)=\int_\Omega h_{\omega, \epsilon}(B)\, d\mathbb P_\epsilon(\omega)
\]
and $\mu:=\mu_0$.
 \end{maintheorem}

  \begin{proof}
We have
\[
\begin{split}
& \int_{\Omega \times M}\Phi \, d\mu_\epsilon-\int_{\Omega \times M}\Phi \, d\mu\\
&=\int_{\Omega'} \left (h_{\omega, \epsilon}-h_\omega)(\Phi(\omega, \cdot) \right)d\mathbb P_{\epsilon} +\int_\Omega h_{\omega}(\Phi(\omega, \cdot))(d\mathbb P_{\epsilon}-d\mathbb P_0)\\
&=:(I)_\epsilon+(II)_\epsilon.
\end{split}
\]
Notice that
\[
|(II)_\epsilon|\le \|\varphi\|_{\mathfrak F} \cdot \|\mathbb P_\epsilon -\mathbb P_0\|_{\mathfrak F}\to 0 \quad \text{when $\epsilon \to 0$,}
\]
by~\eqref{1211}. 

We now turn to $(I)_\epsilon$. In the case of distributions, \eqref{distrib} implies that 
\[
|(h_{\omega, \epsilon}-h_\omega)(\Phi(\omega, \cdot))|\le C\theta(\omega)\|h_{\omega, \epsilon}-h_\omega\|_w,
\]
while in the second case, we have 
\[
|(h_{\omega, \epsilon}-h_\omega)(\Phi(\omega, \cdot))|\le \theta(\omega)\|h_{\omega, \epsilon}-h_\omega\|_{L^\infty(m)}\le  \theta(\omega) \|h_{\omega, \epsilon}-h_\omega\|_w.
\]
Consequently, using Theorem~\ref{T1}, we have
\[
\begin{split}
|(I)_\epsilon|
&\le C\int_{\Omega'}\theta(\omega)\|h_{\omega, \epsilon}-h_\omega\|_{w}\, d\mathbb P_\epsilon(\omega)\\
&\le DR\|\theta\|_{\mathfrak F}(|\epsilon|^\zeta+(\dist_{C^0}(\sigma^{-1}, \sigma_\epsilon^{-1}))^{\alpha\zeta}),
\end{split}
\]
which also goes to $0$ when $\epsilon \to 0$ by the second equality in~\eqref{1211}. The proof of the theorem is complete.
 \end{proof}

 \section{Linear response}\label{LRsec}
In this section, we assume that $\Omega$ is also a compact Riemannian manifold and prove the differentiability of $h_{\om,\epsilon}$ (quenched) and $\mu_\epsilon$ (annealed) at $\epsilon=0$. 
We also consider a third Banach space  $\mathcal B_{ss}$ such that  $\mathcal B_{ss}$ is embedded in $\mathcal B_{s}$. In addition, we suppose that 
\[
 \| \cdot \|_s \le \| \cdot \|_{ss}.
\]
Moreover, we suppose that $\psi$ acts as a bounded functional on $\mathcal B_{ss}$. Let $\mathcal B_s^0$ denote the closed subspace of $\mathcal B_s$ consisting of $h\in \mathcal B_s$ with $\psi(h)=0$.
  \subsection{Quenched linear response}

  We now establish a quenched linear response result.

\begin{maintheorem}\label{T3}
Let the assumption of Theorem~\ref{T1} be in force.  In addition, assume that:
\begin{enumerate}

\item there is a constant $K>0$ such that for every $\omega\in\Omega_0$,
\begin{equation}\label{9:29}
\sup_j\|\mathcal L_\omega^j\|_{s}\leq K;
\end{equation}

\item $h_{\omega}\in \mathcal B_{ss}$ for $\omega \in \Omega_0$ and there is a constant $C>0$ such that $\|h_{\om}\|_{ss}\leq C$ for every $\om\in\Omega_0$;
\item 
for every function $\gamma\colon I\to\Omega$ differentiable at $\epsilon=0$ there exists a bounded operator $\frac{\partial \mathcal L_{\gamma(\epsilon), \epsilon}}{\partial \epsilon}\Big \rvert_{\epsilon=0}\colon \mathcal B_{ss}\to \mathcal B_s^0$ such that for  all 
 $g\in \mathcal B_{ss}$ and $\epsilon\in I$,  we have
\begin{equation}\label{816}
\left\|(\mathcal L_{\gamma(\epsilon), \epsilon}-\mathcal L_{\gamma(0),0})g-\epsilon\left(\frac{\partial\mathcal L_{\gamma(\epsilon), \epsilon}}{\partial \epsilon}\Bigg|_{\epsilon=0}\right)g\right\|_{s}\leq C\left(|\epsilon|^{1+\zeta}+d(\gamma(\epsilon),\gamma(0))^{1+\zeta}\right)\|g\|_{ss},
\end{equation}
where $0<\zeta\leq 1$ is the constant from Theorem \ref{T1};

\item 
the maps $\epsilon\to \sigma_\epsilon^{-k}\om$ are differentiable at $\epsilon=0$ and
\begin{equation}\label{Lipp}
\text{Lip}(\sigma^{-k})\leq Ce^{ck}    
\end{equation}
for some $c<\lambda/(1+\zeta)$ and all $k\in\N$;

\item  $\dist_{C^0}(\sigma^{-1}, \sigma_\epsilon^{-1})=O(\epsilon)$.
\end{enumerate}

Set 
$$
\Lambda_{\om,j}=\frac{\partial \mathcal L_{\sigma^{-(j+1)}_\epsilon \omega, \epsilon}}{\partial \epsilon}\Bigg|_{\epsilon=0},
$$
and suppose that 
\begin{equation}\label{unifb}
\sup_{\omega, j}\|\Lambda_{\omega, j}\|<+\infty.
\end{equation}
Then the series 
$$
\Gamma_{\om}=\sum_{j=0}^\infty\mathcal L_{\sigma^{-j}\omega}^j\Lambda_{\om,j}h_{\sigma^{-(j+1)}\omega}
$$
converges in the norm $\|\cdot\|_{w}$, its $\|\cdot\|_w$ norm is uniformly bounded in $\omega$ (with $\om\in\Omega_0$) and for every $\om\in\Omega_\epsilon\cap\Omega_0$, 
$$
\left\|h_{\omega, \epsilon}-h_\omega-\epsilon\Gamma_\omega\right\|_w=O(|\ln |\epsilon|| \cdot |\epsilon|^{1+\zeta}).
$$
\end{maintheorem}

 \begin{proof}
Recall (see~\eqref{1223}) that for every $\om$,
 $$
d(\sigma_{\epsilon}^{-j}\omega,\sigma^{-j}\omega)\leq C'e^{cj}\dist_{C^0}(\sigma^{-1}, \sigma_\epsilon^{-1})=O(e^{cj}|\epsilon|).
$$
 Applying~\eqref{816} for $\gamma(\epsilon)=\sigma_\epsilon^{-(j+1)}\omega$, we have that uniformly in $\omega$,
\begin{equation}\label{1340}
\left\|(\mathcal L_{\sigma^{-(j+1)}_\epsilon \omega, \epsilon}-\mathcal L_{\sigma^{-(j+1)}\omega})h_{\sigma^{-(j+1)}\omega}-\epsilon \Lambda_{\om,j}h_{\sigma^{-(j+1)}\omega}\right\|_{s}=O(e^{(1+\zeta)cj}|\epsilon|^{1+\zeta}).
\end{equation}
Next, using~\eqref{1314} we have
\[
h_{\omega, \epsilon}-h_\omega=\sum_{j=0}^\infty \mathcal L_{\sigma_\epsilon^{-j}\omega, \epsilon}^j(\mathcal L_{\sigma^{-(j+1)}_\epsilon \omega, \epsilon}-\mathcal L_{\sigma^{-(j+1)}\omega})h_{\sigma^{-(j+1)}\omega}
\]
\[
=\epsilon\sum_{j=0}^\infty\mathcal L_{\sigma_\epsilon^{-j}\omega, \epsilon}^j\Lambda_{\om,j}h_{\sigma^{-(j+1)}\omega}+O_w(|\epsilon|^{1+\zeta}),
\]
where $\|O_w(|\epsilon|^{1+\zeta})\|_w=O(|\epsilon|^{1+\zeta})$. Here we used (recall $(1+\zeta)c<\lambda$) that 
\[
\begin{split}
    &\left \|\sum_{j=0}^\infty \mathcal L_{\sigma_\epsilon^{-j}\omega, \epsilon}^j(\mathcal L_{\sigma^{-(j+1)}_\epsilon \omega, \epsilon}-\mathcal L_{\sigma^{-(j+1)}\omega}-\epsilon \Lambda_{\omega, j})h_{\sigma^{-(j+1)}\omega}\right \|_w\\
    &\le \sum_{j=0}^\infty \left \| \mathcal L_{\sigma_\epsilon^{-j}\omega, \epsilon}^j(\mathcal L_{\sigma^{-(j+1)}_\epsilon \omega, \epsilon}-\mathcal L_{\sigma^{-(j+1)}\omega}-\epsilon \Lambda_{\omega, j})h_{\sigma^{-(j+1)}\omega}\right \|_w \\
    &\le \sum_{j=0}^\infty Ce^{-\lambda j}\left \| (\mathcal L_{\sigma^{-(j+1)}_\epsilon \omega, \epsilon}-\mathcal L_{\sigma^{-(j+1)}\omega}-\epsilon \Lambda_{\omega, j})h_{\sigma^{-(j+1)}\omega}\right \|_w =O(|\epsilon|^{1+\zeta}),\\
\end{split}
\]
by~\eqref{1340}, \eqref{dec} 
and~\eqref{preserv}. 

Next, set 
$$
\Delta_{\omega,j}(\epsilon)=\mathcal L_{\sigma_\epsilon^{-j}\omega, \epsilon}^j\Lambda_{\om,j}h_{\sigma^{-(j+1)}\omega}.
$$
Then 
$$
\|\Delta_{\omega,j}(\epsilon)\|_w\leq Ce^{-\lambda j}.
$$
Now, let us take $N\in\mathbb N$. Then
$$
\sum_{j\geq N}\|\Delta_{\omega,j}(\epsilon)\|_w\leq C'e^{-N\lambda}
$$
where $C'>0$ is a constant.
On the other hand, by \eqref{triple} and \eqref{unifb} there is a  constant $C''>0$ such that
$$
\sum_{j<N}\|\Delta_{\omega,j}(\epsilon)-\Delta_{\omega,j}(0)\|_w\leq 
C''N|\epsilon|^{\zeta}.
$$
Indeed, we have
$$
\mathcal L_{\sigma_\epsilon^{-j}\omega, \epsilon}^j-\mathcal L_{\sigma^{-j}\omega}^j=\sum_{k=0}^{j-1}\mathcal L_{\sigma_\epsilon^{-(j-k-1)}\omega, \epsilon}^{j-k-1}(\mathcal L_{\sigma_\epsilon^{-(j-k)}\om,\epsilon}-\mathcal L_{\sigma^{-(j-k)}\om})\mathcal L_{\sigma^{-j}\omega}^k.
$$
 Using \eqref{dec}, \eqref{triple}, \eqref{1223}, \eqref{unifb}, and the first assumption in the statement of the theorem, we see that
\[
\begin{split}
\left\|\left (\mathcal L_{\sigma_\epsilon^{-j}\omega, \epsilon}^j-\mathcal L_{\sigma^{-j}\omega}^j\right )\Lambda_{\om,j}h_{\sigma^{-(j+1)}\omega}\right \|_w &\le C_0\|\Lambda_{\om,j}\| |\epsilon|^\zeta\sum_{k=0}^{j-1}e^{-(\lambda-c) (j-k-1)}\leq C''|\epsilon|^\zeta
\end{split}
\]
where $C_0,C''$ are two positive constants.
We thus conclude that there is a constant $C'''>0$ such that 
$$
\left\|h_{\om,\epsilon}-h_{\omega}-\epsilon\sum_{j=0}^\infty\Delta_{j,\omega}(0)\right\|_w\leq C'''\left(N|\epsilon|^{1+\zeta}+C'|\epsilon|e^{-N\lambda}\right).
$$
The proof of the theorem is completed by taking
 $N=p|\ln |\epsilon||$ for $p>0$ large enough.
\end{proof}
\begin{remark}\label{Rem1}
  Suppose that \eqref{Lipp} holds with some $c>0$ that is not necessarily smaller than $\lambda/(1+\zeta)$.   Let us take $\alpha\in(0,1]$ so that $c\alpha<\lambda$. Then, by replacing $d(\sigma_\epsilon^{-j}\om,\sigma^{-j}\om)$ and $d(\sigma_\epsilon^{-j}\om,\sigma^{-j}\om)^{1+\zeta}$   with $d(\sigma_\epsilon^{-j}\om,\sigma^{-j}\om)^\alpha$ (using the fact that the metric $d$ is bounded), the proof of Theorem \ref{T3} yields that
  $$
\|h_{\om,\epsilon}-h_\om-\epsilon\Gamma_\om\|_w\leq D|\ln|\epsilon||\left(|\epsilon|^{1+\zeta}+\text{dist}_{C^0}(\sigma_\epsilon^{-1},\sigma^{-1})^\alpha\right)
  $$
  for some constant $D>0$. The case $\alpha=1+\zeta$ puts us in the circumstances of Theorem \ref{T3}, but to get the differentiability of $\epsilon\to h_{\om,\epsilon}$ we only need $\text{dist}_{C^0}(\sigma_\epsilon^{-1},\sigma^{-1})=o(\epsilon^{1/\alpha})$.  
 \end{remark}

 \subsection{Annealed linear response}
 We now give an annealed linear response result.
\begin{maintheorem}\label{T4}
Let the assumptions of Theorems \ref{T2} and \ref{T3} be in force. 
Suppose that for every $f\in\mathfrak F$ the map $\epsilon\to \mathbb P_\epsilon(f)$ is differentiable at $\epsilon=0$.
Let $\Phi:\Omega\times M\to\mathbb R$ be a measurable function such that the functions $\varphi,\theta$ and $\phi$ belong to $\mathfrak F$, where $\varphi$ and $\theta$ were defined in Theorem \ref{T2},
$$
\phi(\omega):=\Gamma_\omega(\Phi (\omega, \cdot))
$$
and $\Gamma_\om$ was defined in Theorem \ref{T3} (recall that  $\Gamma_\om=\frac{d h_{\om,\epsilon}}{d\epsilon}\big|_{\epsilon=0}$ and $\sup_{\om\in\Omega'}\|\Gamma_\om\|_w<\infty$).
Then the following version of the chain rule holds true
$$
\lim_{\epsilon\to 0}\frac{\int_{\Omega \times M}\Phi \, d\mu_\epsilon-\int_{\Omega \times M}\Phi \, d\mu}{\epsilon}=
\int_{\Omega}\phi(\omega)\, d\mathbb P_0(\om)+\mathbb P'_0(\varphi).
$$
\end{maintheorem}
\begin{remark}
We note that in the circumstances of Remark \ref{Rem1} we can also prove the annealed linear response when $\text{dist}_{C^0}(\sigma_\epsilon^{-1},\sigma^{-1})=o(\epsilon^{1/\alpha})$.
\end{remark}
\begin{proof}[Proof of Theorem \ref{T4}]
We have 
\[
\begin{split}
& \int_{\Omega \times M}\Phi \, d\mu_\epsilon-\int_{\Omega \times M}\Phi \, d\mu\\
&=\int_{\Omega'} (h_{\omega, \epsilon}-h_\omega)(\Phi(\omega, \cdot))\, d\mathbb P_{\epsilon} +\int_\Omega h_\omega(\Phi(\omega, \cdot)) \,(d\mathbb P_{\epsilon}-d\mathbb P_0)\\
&=:(I)_\epsilon+(II)_\epsilon.
\end{split}
\]
By Theorem~\ref{T3}, 
$$
(I)_\epsilon=\epsilon\int_{\Omega'} \phi\, d\mathbb P_\epsilon+o(\epsilon).
$$
Now, let us write $\mathbb P_\epsilon=\mathbb P_0+O(\epsilon)$. Using that $\phi\in\mathfrak F$, we thus get that 
$$
(I)_\epsilon=\epsilon\int_{\Omega'}\phi\, d\mathbb P_0+o(\epsilon).
$$
Next, let us write $\mathbb P_\epsilon(\cdot)-\mathbb P_0(\cdot)=\epsilon\mathbb P_0'(\cdot)+o(\epsilon)$. Hence,
$$
(II)_\epsilon=\epsilon \mathbb P_0'(\varphi)+o(\epsilon).
$$
The result follows by combining the above estimates.
\end{proof}
\section{Regularity of the equivariant measures with respect to $(\om,\epsilon)$}\label{RegSec}

Let us fix some $s\in\N$ and suppose that $\mathfrak F=C^s(\Omega)$ (we continue to assume that $\Omega$ is a compact Riemannian manifold).
Let $(\mathcal B^{(i)},\|\cdot\|_i), \ i\leq s$ be a decreasing finite sequence of Banach spaces of distributions (or functions on $M$, depending on the case)  and $\|\cdot\|_{i}\leq \|\cdot\|_{i+1}$.

\begin{maintheorem}\label{RegThm}
Suppose: 
\begin{enumerate}
    \item \eqref{dec} holds with respect to the norm on $\mathcal B^{(1)}$;
    
    \item The operator norm of $\mathcal L_{\om,\epsilon}^n$ is uniformly bounded in $\om,n,\epsilon$ for each one of the spaces $\mathcal B^{(i)}$;

    \item $(\om,\epsilon)\to\mathcal L_{\omega,\epsilon}$ is of class $C^s$ when viewed as a map to $\text{Hom}(\mathcal B^{(i)},\mathcal B^{(i-1)})$;

    \item  $\sup_{\epsilon\in I}\|\sigma_{\epsilon}^{-j}(\cdot)\|_{C^s}=O(e^{cj})$ for $c<\lambda/s^2$.

    \item $\sup_{\om\in\Omega}\sup_{\epsilon\in I}\|h_{\om,\epsilon}\|_{\mathcal B^{(1)}}<\infty$.
\end{enumerate}
Then for every $\delta>0$ the map $(\om,\epsilon)\to h_{\omega,\epsilon}$ is of class $C^{s-\delta}$ when viewed as map to $\mathcal B^{(1)}$.
\end{maintheorem}
\begin{remark}
In Section \ref{SecDiff}, we will use only with $\epsilon=0$. Thus, in the above theorem we can take $I=\{0\}$, namely assume only that $\om\to\mathcal L_{\om,0}$ is of class $C^s$ and $\|\sigma^{-j}\|_{C^s}=O(e^{cj})$.
\end{remark}
\begin{proof}[Proof of Theorem \ref{RegThm}]
Let $u\in \mathcal B^{(s)}$ be such that $\psi(u)=1$.
Since $h_{\om,\epsilon}$ is a uniform in $\epsilon$  and $\omega$ limit of $\mathcal L_{\sigma_\epsilon^{-n}\om,n}^n u$ in the norm $\|\cdot\|_{\mathcal B^{(1)}}$ 
it is enough to show that the $C^r$ norm of $(\om,\epsilon)\to \mathcal L_{\sigma_\epsilon^{-n}\omega,\epsilon}^nu$ is uniformly bounded in $n$ and $\epsilon$. 
Denote by $\mathcal L_{\sigma_\epsilon^{-j}\omega,\epsilon,k}$ the $k$-th derivative of $(\om,\epsilon)\to \mathcal L_{\sigma_\epsilon^{-j}\omega,\epsilon}$. Then
  the $s$-th derivative of  $(\om,\epsilon)\to \mathcal L_{\sigma_\epsilon^{-n}\omega,\epsilon}^n$  is given by 
$$
\sum_{d=1}^{r}\,\sum_{A\subset N_n: |A|=d}\,\sum_{\sum_{i\in A} k_i=s}\,\prod_{j=1}^{n}\mathcal L_{\sigma_\epsilon^{-j}\omega,\epsilon,k_j}
$$
where $\mathbb N_n=\{1,2,\ldots ,n\}$, $k_i\in\N$ and for a given $A$ we set $k_i=0$ for $i\not\in A$. Moreover, $|A|$ denotes the number of elements of $A$.
Now, since 
$$
\psi(\mathcal L_{\sigma_\epsilon^{-j}\omega,\epsilon}\nu)=\psi(\nu)
$$
for all  $\nu$, 
for all $k>0$  we have 
$$
\psi(\mathcal L_{\sigma_\epsilon^{-j}\omega,\epsilon,k}\nu)=0.
$$
Thus, the image of $\mathcal L_{\sigma_\epsilon^{-j}\omega,\epsilon,k}$ is contained in the kernel of $\psi$.

Next, let us fix $1\leq d\leq s$, $A\subset \mathbb N_n$ with $|A|=d$ and positive integers $k_i, i\in A$ with $\sum_{i}k_i=s$. Let $\ell_A$ be the maximum of $A$. Then 
the total length of the blocks $B$ of operators in the composition $\prod_{j=1}^{n}\mathcal L_{\sigma_\epsilon^{-j}\omega,\epsilon,k_j}$ which appear to the left of $\mathcal L_{\sigma_\epsilon^{-\ell_A}\omega,\epsilon ,k_{\ell_A}}$ for which $k_j=0$ for every $j\in B$ is at least $\ell_A-s$. Moreover, the norm of each $\mathcal L_{\sigma_\epsilon^{-j}\omega,\epsilon,k_j}$ when $k_j>0$ is of order $O(e^{cj})=O(e^{c\ell_A})$.
Thus, by relying on~\eqref{dec} applied with the maximal block to the left of $\ell_A$ which contains only $k_j$'s with $k_j=0$ we see that  there exists a constant $C_s>0$ such that
$$
\left\|\prod_{j=1}^{n}\mathcal L_{\sigma_\epsilon^{-j}\omega,\epsilon,k_j}u\right\|_{\mathcal B^{(1)}}\leq C_se^{-(\lambda/s-cs)\ell_A}.
$$
Therefore, there exists a constant $C_s'>0$ such that
$$
\left\|
\sum_{d=1}^{r}\,\sum_{A\subset N_n: |A|=d}\,\sum_{\sum_{i\in A} k_i=s}\,\prod_{j=1}^{n}\mathcal L_{\sigma_\epsilon^{-j}\omega,\epsilon,k_j}
u\right\|_{\mathcal B^{(1)}}\leq C_s'\sum_{\ell=1}^n \ell^{s-1}e^{-\ell(\lambda/s-cs)}.
$$
\end{proof}

\section{Differentiability of the asymptotic moments in the central limit theorem: a direct approach  when $\mathfrak F=C^{s-\delta}(\Omega)$}\label{SecDiff}
Let $T_{\om,\epsilon}:M\to M$ be measurable maps having $\mathcal L_{\omega,\epsilon}$ as their dual operators with respect to $\psi$, namely
$$
\psi((\mathcal L_{\omega,\epsilon} g)\cdot h)=\psi(g\cdot (f\circ T_{\omega,\epsilon})).
$$
Let $f_\epsilon:\Omega\times M\to \mathbb R$ be a measurable function. In the expanding case (i.e., when $\|\cdot\|_\infty\leq\|\cdot\|_{w}$) we assume that $f_\epsilon(\om,\cdot)\in \mathcal B_w$ and $\sup_{\epsilon,\om}\|f_\epsilon(\om,\cdot)\|_w<\infty$ while in the hyperbolic case (i.e., when $\mathcal B_w$ are distributions or order $r$) we suppose that $f_\epsilon(\om,\cdot)\in C^r(M)$ and $\sup_{\epsilon,\om}\|f_\epsilon(\om,\cdot)\|_{C^r}<\infty$.
In the latter case, we also assume that the multiplication of a function in  $C^r(M)$ induces a bounded linear operator on  $\mathcal B_w$. In the former case, we assume that  multiplication of a function in  $\mathcal B_w$ induces a bounded linear operator on  $\mathcal B_w$.  In both cases, it follows that the operators $\mathcal L_{\omega,\epsilon,z}, z\in\mathbb C$ given by 
$$
\mathcal L_{\omega,\epsilon,z}(g)=\mathcal L_{\omega,\epsilon}(g e^{zf_{\omega,\epsilon}})
$$
are uniformly analytic in $z$ when viewed as linear operators on $\mathcal B_w$. Thus, by applying \cite[Theorem D.2]{DolgHaf PTRF 2025} to the space 
$\mathcal B_w$, it follows that the conditions in  \cite{HafMom} are in force, and so  
for every integer $k\geq 2$  and $\epsilon\in I$ there exists a number $M_{k,\epsilon}$ such that for every $\epsilon\in I$ and for $\mathbb P_\epsilon$-a.e. $\om$ we have
$$
M_{k,\epsilon}=\lim_{n \to\infty}\frac{1}{n^{k/2}}\mathbb E_{h_{\omega,\epsilon}}[(S_{n,\epsilon}^\om f-\mu_{\omega,\epsilon}(f_{\omega,\epsilon}))^k].
$$
with 
$$
S_{n, \epsilon}^\om f=\sum_{j=0}^{n-1}f_{\sigma_\epsilon^j\om,\epsilon}\circ T_{\om, \epsilon}^j
$$
where $f_{\om,\epsilon}(\cdot)=f_\epsilon(\omega,\cdot)$ and
$$
T_{\om, \epsilon}^j=T_{\sigma_\epsilon^{j-1}\omega,\epsilon}\circ\ldots\circ T_{\sigma_\epsilon\omega,\epsilon}\circ T_{\omega,\epsilon}.
$$

Next, by using an approximation argument together with the validity of the CLT (which follows from the arguments in \cite{HK} or \cite{TAMS}) and the existence of the asymptotic moments we get, when $M_{2,\epsilon}>0$, 
$$
\lim_{n\to\infty}(nM_{2,\epsilon})^{-k/2}\mathbb E_{h_{\om,\epsilon}}[(\bar S_{n,\epsilon}^\om f)^k]=\mathbb E[Z^k]
$$
where $Z$ is a standard normal random variable. Thus, 
$$
M_{k,\epsilon}=M_{2,\epsilon}^{k/2}\mathbb E[Z^k]
$$
and so it is enough to show that $\epsilon\to M_{2,\epsilon}$ is differentiable at $\epsilon=0$ and to compute its derivative.

Denote 
$$
\bar f_\epsilon(\om,\cdot)=\bar f_{\om, \epsilon}:=f_{\om,\epsilon}-h_{\om, \epsilon}(f_{\om,\epsilon}).
$$
Then
\begin{equation}\label{Var form}
\begin{split}
\Sigma_{\epsilon}^2=:M_{2,\epsilon}&=\int_\Omega h_{\om, \epsilon}(\bar f_{\om, \epsilon}^2)\,  d\mathbb P_\epsilon+2\sum_{n\geq1}\int_\Omega h_{\om, \epsilon}( \bar f_{\om, \epsilon} (\bar f_{\sigma_\epsilon^n \omega, \epsilon}\circ T_{\om, \epsilon}^n))\,  d\mathbb P_\epsilon \\
&=\int_{\Omega \times M}\bar f_{\epsilon}^2d\mu_\epsilon +2\sum_{n\ge 1}\int_{\Omega \times M} (\bar f_{\epsilon}) \cdot (\bar f_{\epsilon} \circ \tau_\epsilon^n)\, d\mu_\epsilon.
\end{split}
\end{equation}
As in  \cite{TAMS} or \cite{HK}, for  every $\epsilon\in I$ for $\mathbb P_\epsilon$-a.e. $\om$ and we see that $n^{-1/2}S_{n,\epsilon}^\om \bar f_\epsilon$ converges in distribution as $n\to\infty$ to a zero mean normal random variable with variance $\Sigma^2_{\epsilon}$, when considered as a random variable on the probability space $(M, h_{\om,\epsilon})$. Thus, $\Sigma^2_\epsilon$ is the asymptotic variance in the corresponding CLT. 

\begin{maintheorem}\label{DiffThm}
Let conditions of Theorems \ref{T3} and Theorem \ref{RegThm} (with $I=\{0\}$ and be in force. Assume also that $\|\cdot\|_w\leq \|\cdot\|_{\mathcal B_1}$. Further assume that 
$$
\sup_{\epsilon\in I}\|\sigma_{\epsilon}^{j}(\cdot)\|_{C^s}=O(e^{cj})
$$
 for $c<\lambda/s$. Let $\mathfrak F=C^{s-\delta}(\Omega)$ and suppose that,
$$
\sup_{g\in\mathfrak F: \|g\|_{\mathfrak F}\leq 1}|(\mathbb P_{\epsilon}-\mathbb P_0-\epsilon\mathbb P'_0)(g)|=o(\epsilon).
$$
Further assume that $(\om,\epsilon)\to f_{\om,\epsilon}$ is of class $C^{s-\delta}$ (as a map to either $\mathcal B_w$ or $C^r(M)$).

Then, the
function $\epsilon\to \Sigma_\epsilon^2$ is differentiable at $\epsilon=0$. In addition, 
$$
d:=\frac{d\Sigma_{\epsilon}^2}{d\epsilon}\Big|_{\epsilon=0}
$$
is given by differentiating each one of the summands in~\eqref{Var form} separately. 
\end{maintheorem}

\begin{proof}[Proof of Theorem \ref{DiffThm}]

Let us first deal with the  term  
$$
d_0(\epsilon):=\int_{\Omega \times M}\bar f_\epsilon^2 \, d\mu_\epsilon.
$$
We have 
\[
\begin{split}
(d_0(\epsilon)-d_0(0))/\epsilon &=
\int_\Omega \epsilon^{-1}(h_{\om, \epsilon}-h_{\om, 0})(\bar f_{\om, \epsilon}^2)\, d\mathbb P_\epsilon(\om)+\int_\Omega h_{\om, 0}(\epsilon^{-1} (\bar f_{\om, \epsilon}^2-\bar f_{\om, 0}^2))\, d\mathbb P_\epsilon(\om) \\
&+\int_\Omega h_{\om, 0}(\bar f_{\om, 0}^2)\, \epsilon^{-1}(d\mathbb P_\epsilon(\om)-d\mathbb P_0(\omega))=:I_1(\epsilon)+I_2(\epsilon)+I_3(\epsilon).
\end{split}
\]
The limit as $\epsilon\to 0$ of $I_1(\epsilon)+I_2(\epsilon)$ is computed similarly to \cite{DS}.  More precisely, we have 
\[
\left |I_1(\epsilon)-\int_\Omega \Gamma_\omega(\bar f_{\omega, \epsilon}^2)\, d\mathbb P_\epsilon(\omega)\right |\le C|\epsilon|,
\]
where $C>0$ is independent of $\epsilon$ and $\omega$  and $\Gamma_\omega=\frac{dh_{\omega, \epsilon}}{d\epsilon}\big \rvert_{\epsilon=0}$.
Writing 
\[
\bar f_{\omega, \epsilon}-\bar f_{\omega, 0}=(f_{\omega, \epsilon}-f_{\omega, 0})+h_{\omega, \epsilon}(f_{\omega, 0}-f_{\omega, \epsilon})+(h_{\omega, 0}-h_{\omega, \epsilon})(f_{\omega, 0}),
\]
we see that 
\[
\lim_{\epsilon \to 0}\left |\int_\Omega \Gamma_\omega(\bar f_{\omega, \epsilon}^2)\, d\mathbb P_\epsilon(\omega)-\int_\Omega \Gamma_\omega(\bar f_{\omega, 0}^2)\, d\mathbb P_\epsilon(\omega)\right |=0.
\]
Consequently, 
\[
\lim_{\epsilon \to 0}I_1(\epsilon)=\int_\Omega \Gamma_\omega(\bar f_{\omega, 0}^2)\, d\mathbb P_0(\omega).
\]
Similarly,
\[
\begin{split}
\lim_{\epsilon \to 0}I_2(\epsilon) &=2\int_\Omega h_{\omega, 0}\left (\frac{df_{\omega, \epsilon}}{d\epsilon}\bigg \rvert_{\epsilon=0}\bar f_{\omega, 0}\right )\, d\mathbb P_0(\omega)-2\int_\Omega h_{\omega, 0}\left(\frac{df_{\omega, \epsilon}}{d\epsilon}\bigg \rvert_{\epsilon=0}\right )h_{\omega, 0}(\bar f_{\omega, 0})\, d\mathbb P_0(\omega)\\
&\phantom{=}-2\int_\Omega \Gamma_\omega(f_{\omega, 0})h_{\omega, 0}(\bar f_{\omega, 0})\, d\mathbb P_0(\omega).
\end{split}
\]
The existence of the limit of $I_3(\epsilon)$ when $\epsilon\to 0$ relies on the linear response assumption on $\epsilon\to \mathbb P_\epsilon$ together the fact that $\om\to h_{\om, 0}(\bar f_{\om, 0}^2)$ is of class $C^{s-\delta}$, which follows from the assumptions regarding $f$ and Theorem \ref{RegThm}.

Now, let us deal with the second term 
$$
\sum_{n\ge 1}\int_{\Omega \times M} h_{\om,\epsilon}(\bar f_{\om, \epsilon}  \cdot ( \bar f_{\sigma_\epsilon^n\om,\epsilon} \circ T_{\om,\epsilon}^n))\, d\mathbb P_\epsilon(\om).
$$
Let us denote 
$$
C_n(\epsilon):=\int_{\Omega \times M} h_{\om,\epsilon}(\bar f_{\om,\epsilon}  \cdot ( \bar f_{\sigma_\epsilon^n\om,\epsilon} \circ T_{\om,\epsilon}^n))\, d\mathbb P_\epsilon(\om)=
\int_\Omega (\mathcal L_{\om,\epsilon}^n(\bar f_{\om,\epsilon}h_{\om,\epsilon}))(\bar f_{\sigma_\epsilon^n\om,\epsilon}) d\mathbb P_\epsilon(\om)
$$
and 
$$
D_n(\epsilon):=\frac{C_n(\epsilon)-C_n(0)}{\epsilon}.
$$
Let us write
$$
D_n(\epsilon)=d_{1,n}(\epsilon)+d_{2,n}(\epsilon)+d_{3,n}(\epsilon)+d_{4,n}(\epsilon)
$$
where 
$$
d_{1,n}(\epsilon)=\int_\Omega (\mathcal L_{\omega,\epsilon}^n(\bar f_{\om, \epsilon}h_{\om, \epsilon}))(\epsilon^{-1}[\bar f_{\sigma_\epsilon^n\om, \epsilon}-\bar f_{\sigma^n\om, 0}]) \, d\mathbb P_\epsilon(\om),
$$
$$
d_{2,n}(\epsilon)=\int_\Omega (\mathcal L_{\om, \epsilon}^n(\epsilon^{-1}(\bar f_{\om, \epsilon}h_{\om, \epsilon}-\bar f_{\om, 0}h_{\om, 0}))(\bar f_{\sigma^n\om, 0})\, d\mathbb P_\epsilon(\om),
$$
$$
d_{3,n}(\epsilon)=
\int_\Omega ([\epsilon^{-1}(\mathcal L_{\om, \epsilon}^n-\mathcal L_{\omega}^n)](h_{\om, 0}\bar f_{\om, 0}))(\bar f_{\sigma^n\om, 0})\, d\mathbb P_\epsilon(\om)
$$
and 
$$
d_{4,n}(\epsilon)=
\int_\Omega (\mathcal L_{\omega}^n(\bar f_{\om, 0}h_{\om, 0}))(\bar f_{\sigma^n\om, 0})\, \epsilon^{-1}(d\mathbb P_\epsilon(\om)-d\mathbb P_0(\om)).
$$
The terms $d_{2,n}(\epsilon)$ and $d_{3,m}(\epsilon)$ can be analyzed similarly to  \cite[proof of Theorem 17]{DS}.
Let us now analyze the term $d_{4,n}(\epsilon)$. We first note that by the assumptions of the theorem together with  Theorem \ref{RegThm} and an modification of the arguments in the proof of Theorem \ref{RegThm} the $C^{r-\delta}$ norms of the functions $\om\to (\mathcal L_{\omega}^n(\bar f_{\om, 0}h_{\om, 0}))(\bar f_{\sigma^n\om, 0})$ are  uniformly bounded. Thus, we can differentiate $d_{4,n}$ at $\epsilon=0$ and sum up the derivatives, using that $d_{4,n}(\epsilon)=O(e^{-\lambda n})$ uniformly in $\epsilon$.

Next, let us analyze the terms $d_{1,n}(\epsilon)$. Let us first show that 
\begin{equation}\label{Show}
\sup_{\epsilon\in I}|d_{1,n}(\epsilon)|=O(\delta^n)    
\end{equation}
for some $\delta\in(0,1)$. Once this is proven we can just sum up  the limits as $\epsilon$ of each $d_{1,n}(\epsilon)$ and get the desired result. To prove \eqref{Show}, by the mean value theorem we have 
$$
\epsilon^{-1}[\bar f_{\sigma_\epsilon^n\om, \epsilon}-\bar f_{\sigma^n\om, 0}]=O_w(1+d(\sigma_\epsilon^n\om,\sigma^n\om)).
$$
Now, arguing like in \eqref{1223} we see that 
$$
d(\sigma_\epsilon^n\om,\sigma^n\om)=O(e^{cn})
$$
uniformly in $\epsilon$ and $\om$. Using \eqref{dec} and that $c<\lambda$ we obtain \eqref{Show}.
\end{proof}

\section{Examples of maps and varying base maps}\label{Examples}
\subsection{The maps $T_{\om,\epsilon}$}

\subsubsection{Expanding maps}

\paragraph{\textbf{The maps and their transfer operator cocycles: condition \eqref{dec}.}}
Let $I$ be an open interval around $0$ in $\mathbb R$.
Let  $T_{\omega,\epsilon} \colon M \to M$, $\omega \in \Omega, \epsilon\in I$ be a collection of non-singular  transformations (i.e.\ $m\circ T_{\omega,\epsilon}^{-1}\ll m$ for each $\omega$) acting   on $M$.
 Each transformation $T_{\omega,\epsilon}$ induces the corresponding transfer operator $\mathcal L_{\omega,\epsilon}$ acting on $L^1(M, m)$ and defined by the following duality relation
\[
\int_M(\mathcal L_{\omega,\epsilon} \phi)\psi \, dm=\int_M\phi(\psi \circ T_{\omega,\epsilon})\, dm, \quad \phi \in L^1(M, m), \ \psi \in L^\infty(M, m).
\]

Let $\alpha\in(0,1]$ and denote by $\var_\alpha(g)\in [0, \infty]$ the H\"older constant corresponding to the exponent $\alpha$ of a function $g:M\to\mathbb R$. Denote 
$$
\|g\|_\alpha=\|g\|_\infty+\var_\alpha(g).
$$
\begin{definition}[Admissible cocycle] \label{def:admis}
We call the family $(\mathcal L_{\omega, \epsilon})_{\omega \in \Omega}$, $\epsilon \in I$ of  transfer operator cocycles  uniformly   
admissible if the following conditions hold:
\begin{enumerate}
\item \label{cond:unifNormBd}
there exists $K>0$ such that for every $\omega \in \Omega$ and $\epsilon\in I$,
\begin{equation*}
 \lVert \mathcal L_{\omega,\epsilon} f\rVert_{\alpha
 } \le K\lVert f\rVert_{\alpha}, \quad \text{for every $f$ such that $\lVert f\rVert_{\alpha}<\infty$;}
\end{equation*}

\item   there are constants $C>0$, $N\in\mathbb N$ and $0<\rho<1$ such that for every $\omega \in \Omega$ and $\epsilon\in I$,
\begin{equation}\label{LY}
\var_\alpha\left(\mathcal L_{\omega,\epsilon}^{n} (f)\right)\leq \rho\var_\alpha(f)+C\|f\|_{L^1(m)}
\end{equation}
for every  $f$ such that $\lVert f\rVert_{\alpha}<\infty$,
where  $\mathcal L_{\omega, \epsilon}^n$ is given by~\eqref{compos};

\item  \label{Min} there exist $c>0$ such that for each $a>0$  and any sufficiently large $n\in \mathbb N$,
\begin{equation*}
\essinf  \mathcal L_{\omega,\epsilon}^{n} f\ge c \lVert f\rVert_{L^1(m)}, \quad \text{for every $f\in C_a, \epsilon\in I, \omega\in\Omega$,}
\end{equation*}
where $C_a:=\{ f\in L^1(m) : f\ge 0 \text{ and } \var_\alpha(f)\le a\int f\, dm \}.$
\end{enumerate}
\end{definition}
We refer to \cite{DolgHaf PTRF 2025} for examples of maps satisfying the above conditions. 
This includes already expanding maps $T_{\om,\epsilon}\colon M\to M$ satisfying the conditions of \cite[Ch. 5]{HK} with the same constants. That is, let $\rho$ be a Riemannian metric on $M$. 
This means that there exist constants $\gamma>1$, $\xi>0$, and $D,N_0\in\mathbb N$ such that:
\begin{enumerate}
 \item[(i)] \textbf{Local expansion. }For every $\om\in\Omega$, $\epsilon\in I$, and all $x,y\in M$ such that $\rho(x,y)\leq\xi$ we can write 
$$
(T_{\om, \epsilon})^{-1}(\{x\})=\{x_1,\ldots ,x_k\},\,\,\, (T_{\om, \epsilon})^{-1}(\{y\})=\{y_1, \ldots ,y_k\}
$$
with $k\leq D$ and for each $1\le i\le k$ we have 
$$
\rho(x_i,y_i)\leq \gamma^{-1}\rho(x,y).
$$
\item[(ii)]\textbf{Covering. } For every $x\in M$, $\om\in\Omega$, and $\epsilon\in I$ we have 
$$
T_{\om,\epsilon}^{N_0}(B(x,\xi))=M,
$$
where $B(x,\xi)$ is the open ball around $x$ with radius $\xi$.
\end{enumerate}
Note that when $\xi>\text{diam}(M)$ we can drop the assumption that the degree is finite, that is, we allow the inverse images of the singletons to be countable. Moreover, the operator $\mathcal L_{\om,\epsilon}$ is given by 
$$
\mathcal L_{\om,\epsilon}g(x)=\sum_{y_{\om,\epsilon}}J(y_{\om,\epsilon}(x))g(y_{\om,\epsilon}(x))
$$
where the sum is over all the inverse branches of $T_{\omega,\epsilon}$ and $J(y)$ is the Jacobian of $y$. In the above case 
we take $\alpha$ such that $\sup_{\omega,\epsilon}\|\ln(J(T_{\om,\epsilon}))\|_{\alpha}<\infty$ (assuming such one exists). Here $J(T)$ is the Jacobian of a map $T\colon M\to M$.

Arguing like in the proof of \cite[Theorem 2.4]{DolgHaf PTRF 2025} we see that there exist strictly positive functions $v_{\omega,\epsilon}\colon M\to\mathbb R$ with uniformly bounded $\|\cdot\|_\alpha$ norms such that $\mathcal L_{\omega,\epsilon}v_{\omega,\epsilon}=v_{\sigma_\epsilon\omega,\epsilon}$ for $\omega \in \Omega$ and $\epsilon \in I$. Thus, the underlying measures are $h_{\omega,\epsilon}=v_{\omega,\epsilon}\,dm$.
Moreover, there exist constants $C, \lambda>0$ such that for all $\epsilon\in I$ and $\omega\in\Omega$ we have
\begin{equation}\label{decAp}
\|\mathcal L_{\omega,\epsilon}^n-m\otimes v_{\sigma_\epsilon^n\omega,\epsilon}\|_{\alpha}\leq Ce^{-\lambda n}.  
\end{equation}
This shows that \eqref{dec} holds  with $\|\cdot\|_w=\|\cdot\|_{\alpha}$.

\paragraph{\textbf{Continuity and differentiability of the transfer operator with respect to $\epsilon$.}}

First, we discuss the uniform boundedness of the functions $v_{\om,\epsilon}$ in appropriate norms. 
\begin{lemma}
Take $m\in\mathbb N$, $m\geq 2$.
Suppose that there are constants $c>0$ and $\gamma>1$ such that 
$$
\sup_{\omega\in\Omega}\sup_{\epsilon\in I}\|D(T_{\omega,\epsilon})\|_{C^m}\leq c
$$
and that
for all inverse branches $y$ of $T_{\omega,\epsilon}$ we have $\|Dy\|_{\infty}\leq \gamma^{-1}$.
 Then for every $\delta\in(0,1)$\, \eqref{9:29} holds with the norm $C^m$ and,
$$
\sup_{\omega\in\Omega}\sup_{\epsilon\in I}\|v_{\omega,\epsilon}\|_{C^{m-\delta}}<\infty.
$$
\end{lemma}
\begin{proof}
Since $v_{\om,\epsilon}$ is a uniform limit of $\mathcal L_{\sigma_{\epsilon}^{-n}\omega,\epsilon}^n\textbf{1}$, we see that $\sup_{\om,\epsilon}\|h_{\om,\epsilon}\|_{C^{m-\delta}}<\infty$ if $\mathcal L_{\sigma_{\epsilon}^{-n}\omega,\epsilon}^n\textbf{1}$ are uniformly bounded in the $\|\cdot\|_{C^{m}}$ norm for some $\delta>0$.  

Now, arguing like in the proof of \cite[Lemmata 19-22]{DH CMP} and \cite[Corollary 23]{DH CMP} in the case that all the random variables appearing there are bounded, we see that 
$$
\|\mathcal L_{\sigma_{\epsilon}^{-n}\omega,\epsilon}^n\|_{C^m}\leq C_m\|\mathcal L_{\sigma_{\epsilon}^{-n}\omega,\epsilon}^n\textbf{1}\|_{\infty}.
$$
Finally, by the Lasota-Yorke inequality \eqref{LY} we see that $\|\mathcal L_{\sigma_{\epsilon}^{-n}\omega,\epsilon}^n\textbf{1}\|_{\infty}$ is uniformly bounded. 
\end{proof}

Next,  for $\beta\leq 1$ denote $\|\cdot\|_{C^\beta}=\|\cdot\|_\beta$ (with $C^0$ denoting the space of continuous functions). For $\beta>1$ denote by $C^\beta$ the space of all functions $g$ which are $[\beta]$-times differentiable and whose $[\beta]$-th derivative is of class $C^{\beta-[\beta]}$. In that case, let
$$
\|g\|_{C^\beta}=\|g\|_{C^{[\beta]}}+\|D^{[\beta]}g\|_{C^{\beta-[\beta]}}.
$$

We need the following result.
\begin{lemma}
Suppose that for all inverse branches $y$ of $T_{\omega,\epsilon}$ we have $\|Dy\|_{\infty}\leq \gamma^{-1}$, for some constant $\gamma>1$.
 Take some $\alpha\leq 1$ and $\beta>\alpha$ so that $(\om,\epsilon,x)\to T_{\omega,\epsilon}(x)$ is of class $C^{\beta}$.
 In particular,
$$
d_{C^\beta}(T_{\omega,\epsilon},T_{\omega'})\leq C(|\epsilon|+d(\om,\om')).
$$
Then \eqref{triple} holds with $\|\cdot\|_w=\|\cdot\|_{C^\alpha}$, $\|\cdot\|_{s}=\|\cdot\|_{C^\beta}$, and $\zeta=\beta-\alpha$.   
\end{lemma}
\begin{proof}
Write 
$$
\mathcal L_{\omega,\epsilon}g(x)=\sum_{i}J(y_{\om,\epsilon,i}(x))g(y_{\om,\epsilon,i}(x))
$$
where $y_{\om,\epsilon,i}$ are the inverse branches of $T_{\omega,\epsilon}$ and $J(y)$ is the Jacobian of a function $y$. Then the lemma follows since 
$$
d_{C^\beta}(y_{\om,\epsilon,i},y_{\om',0,i})\leq C'(|\epsilon|+d(\om,\om'))
$$
and $\|\mathcal L_{\omega,\epsilon}\textbf{1}\|_{C^{\beta+1}}$ is uniformly bounded (note that $\mathcal L_{\omega,\epsilon}\textbf{1}=\sum_iJ(y_{\om,\epsilon,i})$). 
\end{proof}

Finally, we need:
\begin{lemma}
Suppose that for all inverse branches $y$ of $T_{\omega,\epsilon}$ we have $\|Dy\|_{\infty}\leq \gamma^{-1}$, for some constant $\gamma>1$.
 Take some $\alpha_0>0$ and $\beta>\alpha_0$ and suppose 
 that $(\om,\epsilon,x)\to T_{\omega,\epsilon}(x)$ is of class $C^{\beta+1}$. In particular,
 $$
d_{C^{\beta+1}}(T_{\omega,\epsilon},T_{\omega'})\leq C(|\epsilon|+d(\om,\om')).
$$
Then \eqref{816} holds with $\|\cdot\|_s=\|\cdot\|_{C^{\alpha_0}}$, $\|\cdot\|_{ss}=\|\cdot\|_{C^{\beta+1}}$, and $\zeta=\beta-\alpha_0$.   
\end{lemma}
\begin{proof}
Write 
$$
\mathcal L_{\gamma(\epsilon),\epsilon}g(x)=\sum_{i}J(y_{\gamma(\epsilon),\epsilon,i}(x))g(y_{\gamma(\epsilon),\epsilon,i}(x)).
$$
 Then
$$
d_{C^{\beta+1}}(y_{\om,\epsilon,i},y_{\om',0,i})\leq C'(|\epsilon|+d(\om,\om'))
$$
and $\|\mathcal L_{\omega,\epsilon}\textbf{1}\|_{C^{\beta+1}}$ is uniformly bounded. Thus, the lemma follows by Taylor expansion when viewed as a function of $\epsilon$. Here we used the fact that if $G:Q\to E$ is of class $C^{r+q}$ for a manifold $Q$, a Banach space $E$, $r\in\mathbb N$ and $q\in(0,1]$, then the Taylor reminder of order $r$ is of size $O(\delta^{r+q})$ when perturbing a point $q_1\in Q$ by a point $q_2\in Q$ such that $\text{dist}(q_1,q_2)\leq \delta$.  
\end{proof}

Next, let us discuss the conditions of Theorem \ref{RegThm}.
We first have the following result whose proof proceeds similarly to the previous lemma.
\begin{lemma}
Suppose that for all inverse branches $y$ of $T_{\omega,\epsilon}$ we have $\|Dy\|_{\infty}\leq \gamma^{-1}$, for some constant $\gamma>1$.
 Take some $\alpha_0>0$, $s\in\mathbb N$ and $\beta>\alpha_0$ and suppose 
 that $(\om,\epsilon,x)\to T_{\omega,\epsilon}(x)$ is of class $C^{\beta+s}$. 
Then $(\om,\epsilon)\to\mathcal L_{\omega,\epsilon}$ is of class $C^s$ when viewed as a map to $\text{Hom}(C^{\beta+s}(M),C^{\alpha_0}(M))$.  
\end{lemma}
Using the above lemma, we can choose chains of spaces $\mathcal B^{(i)}$ of the form $\mathcal B^{(i)}=C^{i(s+1)}(M)$ assuming that  $(\om,\epsilon,x)\to T_{\omega,\epsilon}(x)$ is of class $C^{s^2+1}$. Note that in this case, the operator norms of $ \mathcal L_{\omega,\epsilon}^n$ are uniformly bounded in $\om,\epsilon,n$ when viewed as operators on $\mathcal B^{(i)}$.

\subsubsection{Small random perturbations of an Anosov diffeomorphism}
Let $M$ denote a $C^\infty$ compact and connected Riemannian manifold. Furthermore, let $T$ be
a transitive Anosov diffeomorphism on $M$ of class $C^{t+1}$
for $t>1$. By $\mathcal L_T$ we denote 
the transfer operator associated with $T$. We recall that the action of $\mathcal L_T$ on smooth functions $h\in C^t(M,\mathbb R)$ is given by
$$
 \mathcal L_T h=\bigg{(}\frac{h} {\lvert \det T\rvert}\bigg{)}\circ T^{-1}.
$$
 For $\Delta>0$,  let $B_{C^{t+1}}(T, \Delta)$ denote the ball of radius $\Delta$ centered at $T$ in the $C^{t+1}$-topology. It is well-known that $B_{C^{t+1}}(T, \Delta)$ consists of Anosov diffeomorphisms provided that $\Delta$ is sufficiently small.

 Let $\Omega$ be a compact Riemannian manifold and $I\subset \mathbb R$ be an open interval that contains $0$. Furthermore, take $s>1$ and consider a mapping $\mathbf T\colon \Omega \times I \to C^{t+1}(M, M)$ of class $C^s$ such that $\mathbf T(\omega, 0)\in B_{C^{t+1}}(T, \Delta)$ for $\omega \in \Omega$. Set
 \[
 T_{\omega, \epsilon}:=\mathbf T(\omega, \epsilon), \quad (\omega, \epsilon)\in \Omega \times I.
 \]
Note that by shrinking $I$ if necessary, we have that  $T_{\omega, \epsilon}\in B_{C^{t+1}}(T, \Delta)$ (and thus $T_{\omega, \epsilon}$ is Anosov) for $\omega \in \Omega$ and $\epsilon \in I$.

 The Banach spaces to which our results are applicable belong to the class of anisotropic Banach spaces $\mathcal B^{p, q}$ introduced by Gou\"{e}zel and Liverani~\cite{GL}, where $p\in \mathbb N_0$,  $q>0$ and $p+q<t$. We recall that elements of $\mathcal B^{p, q}=(\mathcal B^{p, q}, \| \cdot \|_{p, q})$ are distributions of the order of most $q$. These Banach spaces are associated with our fixed transitive Anosov diffeomorphism $T$. However, the transfer operators associated with any $S\in B_{C^{t+1}}(T, \Delta)$ are well-defined and bounded linear operators on each of these spaces (provided that $\Delta$ is sufficiently small). We also recall that the action of the transfer operator $\mathcal L_S$ associated with $S$ on $\mathcal B^{p, q}$ is given by 
 \[
 (\mathcal L_Sh)(\varphi)=h(\varphi \circ S), \quad h\in \mathcal B^{p, q}, \ \varphi \in C^q(M, \mathbb C).
 \]

 In order to apply Theorem~\ref{T1}, we can take $\mathcal B_w=\mathcal B^{p, q}$ and $\mathcal B_s=\mathcal B^{p+\ell, q-\ell}$, where $p\ge 1$, $q, \ell \in \mathbb N$, $q>\ell$ and $p+q<t$. For example, if $t>3$ we can choose $\mathcal B_w=\mathcal B^{1, 2}$ and $\mathcal B_s=\mathcal B^{2, 1}$.
 The functional $\psi$ is given by $\psi (h)=h(1)$, where $h(1)$ denotes the evaluation of a distribution $h$ in a constant function $1$.
 Then, shrinking $I$  and $\Delta$ if necessary, we have that all the assumptions of Theorem~\ref{T1} hold with $\Omega_\epsilon=\Omega$. Let us provide the details.  The arguments in the proof of~\cite[Lemma 14]{DS} give that 
 \[
 d_{C^{r+1}}(T_{\omega, \epsilon}, T_{\omega', 0})\le C(d(\omega, \omega')+|\epsilon|),
 \]
 for $\omega \in \Omega$ and $\epsilon \in I$, where $C>0$ is independent of these. This together with~\cite[Lemma 7.1]{GL} implies that~\eqref{triple} holds. The same argument gives that 
 \begin{equation}\label{small1}
 \|(\mathcal L_{\omega, \epsilon}-\mathcal L_T)h\|_{p, q}\le Cd_{C^{r+1}}(T_{\omega, \epsilon}, T)\|h\|_{p+1, q-1}\le C\Delta \|h\|_{p+\ell, q-\ell} \quad \text{for $h\in \mathcal B_s$,}
 \end{equation}
 and 
 \begin{equation}\label{small2}
  \|(\mathcal L_{\omega, \epsilon}-\mathcal L_T)h\|_{p-1, q+1}\le  C\Delta \|h\|_{p, q} \quad \text{for $h\in \mathcal B_w$.}
  \end{equation}
  Due to~\eqref{small2} and 
  since $\mathcal L_T$ has a spectral gap on $\mathcal B_w$ (see~\cite[Theorem 2.3]{GL}), \eqref{dec} follows from~\cite[Proposition 2.10]{CR}. Here we also use that we have uniform Lasota-Yorke inequalities for compositions of transfer operators associated with Anosov maps in $B_{C^{r+1}}(T, \Delta)$ like in~\cite[Eq.(10) and (11)]{TAMS} with $\|\cdot \|_{p-1, q+1}$ instead of $\|\cdot \|_{0, 2}$ and $\|\cdot \|_{p, q}$ instead of $\|\cdot \|_{1, 1}$, respectively.  In addition, the same reasoning gives that~\eqref{dec} is valid with $\|\cdot \|_s$ instead of $\|\cdot \|_w$. This together with~\cite[Proposition 7]{DS} gives for each $\epsilon\in I$ the existence of a (measurable) family $(h_{\omega, \epsilon})_{\omega \in \Omega}\subset \mathcal B_s$ that satisfies the assumptions of Theorem~\ref{T1}. In addition, $h_{\omega, \epsilon}$ is a Borel probability measure on $M$ (see the proof of~\cite[Proposition 3.3]{TAMS}). We also note that~\eqref{distrib} holds for $r= q$.

  Let us discuss the applicability of Theorem~\ref{T3}. For this, we can take $\mathcal B_w=\mathcal B^{p, q}$, $\mathcal B_s=\mathcal B^{p+\ell_1, q-\ell_1}$, and $\mathcal B_{ss}=\mathcal B^{p+\ell_1+\ell_2, q-\ell_1-\ell_2}>0$, where  $p\ge 1$, $q, \ell_1, \ell_2\in \mathbb N$, $\ell_2\ge 2$, $q-\ell_1-\ell_2>0$, and $p+q<t$. For example, if $t>5$ we can choose $\mathcal B_w=\mathcal B^{1, 4}$, $\mathcal B_s=\mathcal B^{2, 3}$, and $\mathcal B_{ss}=\mathcal B^{4, 1}$.
  The assumptions in the statement of Theorem~\ref{T3} that come from Theorem~\ref{T1} can be verified using the same argumentation as in the previous paragraph. Moreover, \eqref{9:29} follows from the strong Lasota-Yorke inequality (as in~\cite[Eq. (11)]{TAMS}) in the pair of norms $(\|\cdot \|_w, \|\cdot \|_s)$. Moreover, for each $\epsilon \in I$, $(h_{\omega, \epsilon})_{\omega \in \Omega}\subset \mathcal B_{ss}$, and we have a uniform bound for $\|h_{\omega, \epsilon}\|_{ss}$ (see~\cite[Proposition 7]{DS}). Finally, \eqref{816} with $\xi=1$ follows by Taylor expansion of $\mathcal L_{\gamma(\epsilon), \epsilon}$ with respect to $\epsilon$ and using the arguments as in the proof of~\cite[Theorem 1]{DS}.

  Let us discuss the applicability of Theorem~\ref{RegThm} in this context. 
  We choose a finite sequence of anisotropic Banach spaces $\mathcal B^{(i)}$, $1\le i \le s$ as follows:  set $\mathcal B^{(1)}:=\mathcal B^{p, q}$, and inductively, if $\mathcal B^{(i)}=\mathcal B^{p_i, q_i}$, then set
 $\mathcal B^{(i+1)}:=\mathcal B^{p_i+r+2, q_i-r-2}$ for $1\le i<s$.  Note that 
  \[
  p_i=(i-1)(r+2)+p \quad \text{and} \quad q_i=q-(i-1)(r+2),
  \]
  for $1\le i \le s$.  Here, $p\ge 1$, $q\in \mathbb N$ with $q>(s-1)(r+2)$ and $p+q<t$.
  We note that ``between'' $\mathcal B^{(i)}$ and $\mathcal B^{(i+1)}$ (including these) we have $r+3$ anisotropic Banach spaces $\mathcal B^{p_i+k, q_i-k}$ for $0\le k \le r+2$. Writing $\omega$ in local coordinates $x\in \mathbb R^d$, $d=\dim M$, and Taylor expanding $\mathcal L_{\omega, \epsilon}$ with respect to $(\omega, \epsilon)$, the arguments as in~\cite[Section 4]{CN} or~\cite[Section 8]{GL} yield (provided $\mathbf T$ is of class $C^{r+2}$) that  $\mathcal L_{\omega, \epsilon}$ is $(r+1)$-times differentiable as a function of $(\omega, \epsilon)$ to $\text{Hom}(\mathcal B^{(i)},\mathcal B^{(i-1)})$ for $1\le i \le r$.
  In particular, it is of class $C^r$.




\subsection{The base maps $\sigma_\epsilon$}
\paragraph{\textbf{Statistical stability.}}
The first condition in \eqref{1211} holds when $\mathfrak F$ is the class of $C^{s}$ functions (for some $s>0$) when $\mathbb P_\epsilon$ are SRB measures of the $C^\infty$ partially hyperbolic maps $\sigma_\epsilon$ considered in \cite{Dol}, see \cite[Theorem 1]{Dol}. In fact, the measures $\mathbb P_\epsilon$ are differential in $\epsilon$.  The second condition in \eqref{1211} just means that that $\sigma_\epsilon^{-1}$ is continuous in $\epsilon$ in the uniform distance, which holds true in the setting of \cite{Dol}.



\paragraph{\textbf{Linear response.}}
The conditions of Theorem \ref{T4} hold when 
 $\sigma_\epsilon(\omega)=g_{\epsilon}\omega$ is a group rotation and $\epsilon\to g_{\epsilon}$ is differentiable in $\epsilon$. Here $\mathbb P_\epsilon=\mathbb P$ is the Haar measure and we work with a left invariant metric, so that all $\sigma_\epsilon^{-1}$ are isometries, and in particular \eqref{Lipp} holds. Since $\mathbb P_\epsilon=\mathbb P$,  we can take any space $\mathfrak F$ such that $\|f\|_{L^1(m)}\leq \|f\|_{\mathfrak F}$ and then take functions $\Phi$ such that the functions $\theta,\varphi,\phi$ in theorems \ref{T2} and \ref{T4} are well defined and belong to $\mathfrak F$. Note that $R$ in Theorem \ref{T2} is finite, since $\|f\|_{L^1(m)}\leq \|f\|_{\mathfrak F}$.

Another example are the small perturbations of circle rotations considered in \cite{GaSo}. Then \eqref{Lipp} holds trivially since $\sigma^{-1}$ is an isometry. Moreover, in the circumstances of  \cite{GaSo} we can take $\mathbb  P_\epsilon(dx)=\rho_\epsilon(x) dx$ such that $(\rho_\epsilon-1)/\epsilon$ converges uniformly (possibly when restricted to a Cantor subset of $I$) to the density $\gamma_0$ of the limiting measure $\mathbb P_{0}'=\gamma_0dx$ as $\epsilon\to 0$. Thus, we can take $\mathfrak F$ to be the class of all bounded measurable functions equipped with the supremum norm.





 \section*{Acknowledgments}
Funded by the European Union-NextGenerationEU-StatDinMatAn-uniri-iz-25-108.

\end{document}